\documentclass[12pt]{amsart}

\usepackage{amsmath,amssymb,amsthm}
\usepackage{mathtools}
\mathtoolsset{showonlyrefs=true}

\usepackage{fullpage}
\usepackage{url}
\usepackage{hyperref}
\usepackage{microtype}

\hypersetup{
  colorlinks=true,
  linkcolor=blue,
  citecolor=blue,
  urlcolor=blue
}

\usepackage{listings}
\usepackage{xcolor}

\newtheorem{theorem}{Theorem}[section]
\newtheorem{lemma}[theorem]{Lemma}
\newtheorem{proposition}[theorem]{Proposition}

\theoremstyle{definition}
\newtheorem{definition}[theorem]{Definition}
\newtheorem{remark}[theorem]{Remark}

\numberwithin{equation}{section}

\newcommand{\F}{\mathbb{F}}
\newcommand{\GL}{\mathrm{GL}}
\newcommand{\SL}{\mathrm{SL}}
\newcommand{\PGL}{\mathrm{PGL}}
\newcommand{\Sym}{\mathrm{Sym}}

\newcommand{\tensor}{\otimes}
\newcommand{\diag}{\mathrm{diag}}

\def\DD{D\kern-.7em\raise0.4ex\hbox{\char '55}\kern.33em}

\title[Spectral Separation for Polynomial Tensor Representations]{Spectral Separation and Eigenvalue Labelling\\ for Polynomial Tensor Representations\\ of General Linear Groups}

\author{\DD\d {\u a}ng V\~o Ph\'uc}
\address{Department of Mathematics, FPT University, Quy Nhon AI Campus, An Phu Thinh New Urban Area, Vietnam}
\email{dangphuc150488@gmail.com}
\thanks{ORCID: \url{https://orcid.org/0000-0002-6885-3996}}


\begin{document}

\begin{abstract}
Let $q=p^f$ be a prime power, let $H \leq \GL_d(q)$ be a subgroup containing a genuine Singer cycle $s$ of order $q^d-1$, and let $W$ be an $\F_qH$--module whose scalar extension is identified with the restriction of an untwisted polynomial tensor representation $\bigotimes_{t=1}^r L(\lambda^{(t)})$ of the algebraic group $\GL_d$. If the factors have polynomial degrees $k_t$ and total degree $K=\sum_t k_t<q-1$, we prove that distinct weights give distinct eigenvalues of $s$ on $W\otimes_{\F_q}\F_{q^d}$. The proof is based on an elementary base-$q$ injectivity lemma: bounded digit vectors determine distinct residues modulo $q^d-1$. When the tensor product is multiplicity-free for the diagonal torus, the Singer cycle therefore has simple spectrum over the splitting field $\F_{q^d}$.

We also record a shifted exponent formula for situations in which the Singer eigenvalue data are acted on by powers of the $q$--Frobenius map. This gives separation of distinct shifted digit vectors under the same bound $K<q-1$. We deliberately distinguish this $q$--Frobenius eigenvalue bookkeeping from the usual Steinberg $p$--Frobenius twists; the latter require separate module-specific analysis.

These results give a uniform spectral explanation for the eigenvalue-separation phenomenon in bounded-degree polynomial tensor representations and isolate the main spectral input needed for rewriting. Motivated by this, we formulate a conditional rewriting framework based on Singer-type elements, compatible base-$q$ eigenvalue labelling, and polynomial-functor combinatorics. The reconstruction of the natural action is reduced to a functor-specific inversion and normalization problem. Finally, the viability of this framework is demonstrated through computational experiments and explicit examples, including a full algebraic reconstruction of the natural action from a strictly multiplicity-free, genuine tensor product representation.
\end{abstract}

\maketitle

\section{Introduction}

\subsection{Constructive recognition of matrix groups}

Constructive recognition of finite groups is a central theme in computational group theory.
Given a finite group $G$ specified in some implicit form, for example as a subgroup of a permutation group or a matrix group, one aims to construct an explicit isomorphism between $G$ and a standard copy of a known abstract group $H$.
This problem has been intensively studied for symmetric and alternating groups, classical groups, and other families of groups of Lie type; see, for example, Beals--Leedham-Green--Niemeyer--Praeger--Seress for symmetric and alternating groups~\cite{BealsLGNPS2003}, and Brooksbank's work on classical groups in their natural representation~\cite{Brooksbank2003}.

In the setting of matrix groups, the natural representation plays a distinguished role.
For a classical group $H$ of dimension $d$ over $\F_q$, a great deal of structure is visible in its action on the natural module $V \cong \F_q^d$.
Explicit recognition algorithms for $H \leq \GL(V)$, in the natural representation, were developed by Brooksbank and others~\cite{Brooksbank2003}, and later extended to the black--box setting and more general classical groups by Dietrich--Leedham-Green--O'Brien~\cite{DietrichLGO2015}.
These algorithms typically assume that the given group acts on a space of dimension $d$ and that the representation is already natural (or close to natural).

In practice, however, matrix groups often arise via representations of dimension $n$ different from $d$, and a central task is to rewrite such a representation in terms of the natural one.
Formally, suppose $G \leq \GL(W)$ is a group generated by a set of matrices $X$ acting irreducibly on an $n$--dimensional $\F_q$--vector space $W$, and that $G$ is known to arise from a classical group $H$ whose natural module has dimension $d$.
The \emph{rewriting problem}, broadly construed, is to recover from the given action on $W$ a natural or projective-natural copy of the underlying degree-$d$ action.
In the present paper we work only at the projective level, and our goal is to construct a homomorphism
\[
  \varphi : G \longrightarrow \PGL_d(q),
\]
equivalent to the natural projective representation of the target subgroup on its natural module, in such a way that $\varphi(g)$ can be effectively computed from the matrix of $g$ on $W$.

\subsection{Previous work on rewriting algorithms}

The first general treatment of rewriting for small-dimensional representations is due to Magaard, O'Brien and Seress~\cite{MagaardOBrienSeress2008}.
They consider the case when $G \cong H$ with $\SL_d(q) \leq H \leq \GL_d(q)$ and $W$ is an irreducible $\F_qG$--module of dimension at most $d^2$.
They develop a Las Vegas polynomial-time algorithm in that setting, based on a detailed analysis of the possible modules of dimension at most $d^2$.

Subsequently, more specialised rewriting algorithms have been developed for particular classes of representations that occur frequently in computational practice.
Corr~\cite{Corr2015} studies the symmetric square representation and analyses a Las Vegas approach to rewriting the representation afforded by $\Sym^2(V)$ to a projective copy of the natural representation. In the preprint cited here, the algorithmic statement is formulated with an additional conditional ingredient; regardless of that algorithmic status, the paper isolates important structural features of the symmetric-square case and shows how such modules can be exploited when they occur as composition factors.

A further step was taken by G\"ul and Ankaral{\i}o\u{g}lu~\cite{GulAnk2016}.
They study the case where $W$ is in a tensor family of \emph{twisted modules} of degree between $d^2$ and $d^3$.
More precisely, they assume $W$ is a twisted tensor product of highest weight modules with highest weights among
\[
  \lambda_1, \lambda_2, \lambda_{d-2}, \lambda_{d-1}, 2\lambda_1, 2\lambda_{d-1}
\]
for a classical group $H$ of type $A$, and they develop a Las Vegas algorithm that rewrites the action on $W$ to a projective action of degree $d$.
Their analysis combines representation theory (via Steinberg's tensor product theorem) with careful, yet ad-hoc, eigenvalue computations for particular tensor products such as
\[
  V \tensor V^{\tau} \tensor V^{\tau^2}, \quad
  V \tensor (\wedge^2 V)^{\tau}, \quad
  V \tensor (\Sym^2 V)^{\tau},
\]
where $\tau$ denotes the relevant field/Frobenius automorphism in their setting.  Some of the highest weights in their list are most naturally interpreted from the special-linear or dual-module viewpoint; they are not all polynomial $\GL_d$--modules of degree at most two under the convention used below.

These contributions fit into the broader matrix group recognition project, which aims to build a general framework for the constructive recognition of finite matrix groups via composition trees and local handlers for particular types of composition factors and modules; see, for example, \cite{Brooksbank2003,DietrichLGO2015} and references therein.

\subsection{Limitations of existing work}

The rewriting algorithms in \cite{Corr2015,GulAnk2016,MagaardOBrienSeress2008} are precise and effective for their intended families of modules, but their scope is restricted in two ways.

First, the work of Magaard--O'Brien--Seress is constrained to representations of dimension at most $d^2$.
Beyond this range, the classification of possible irreducible modules becomes significantly more complex.
Their methods rely on a detailed understanding of the small-degree representation theory of finite groups of Lie type in defining characteristic and do not immediately extend to larger families of modules; see also L\"ubeck's tables and bounds for small-degree representations~\cite{Luebeck2001}.

Second, the twisted-module algorithm of G\"ul--Ankaral{\i}o\u{g}lu~\cite{GulAnk2016} focuses on a fixed list of highest weights of small natural or dual type.
Under the polynomial $\GL_d$ convention used in this paper, however, the dual-side weights such as $\lambda_{d-1}$ and $\lambda_{d-2}$ are not degree-one or degree-two polynomial modules; they correspond instead to exterior powers of the dual or to projectively normalised special-linear weights.
The proofs in~\cite{GulAnk2016} therefore involve module-specific and projective information in addition to case-by-case eigenvalue computations.
While this approach works well for the particular modules considered there, it does not directly generalise to arbitrary polynomial highest weights or more complicated tensor constructions.

On the other hand, the general black--box recognition algorithms for classical groups~\cite{Brooksbank2003,DietrichLGO2015} treat all representations uniformly, without exploiting representation-theoretic structure such as the polynomial degree of highest weights.
This leads to algorithms that are broadly applicable but may be suboptimal on specific families of modules.

\subsection{Contribution of this paper}

The aim of this paper is to isolate a structural condition under which a genuine Singer cycle has separated eigenvalues on bounded-degree polynomial tensor representations, and to explain how this spectral property can be used as the basis of a rewriting strategy.

We work with subgroups $H \leq \GL_d(q)$ that contain a Singer cycle $s \in H$ of order $q^d-1$, and we first treat untwisted polynomial tensor representations
\[
  U = \bigotimes_{t=1}^r L(\lambda^{(t)}),
\]
where each $L(\lambda^{(t)})$ is a polynomial highest-weight module of degree $k_t$.  If the total degree
\[
  K := \sum_{t=1}^r k_t
\]
satisfies $K<q-1$, then a simple base-$q$ injectivity lemma shows that distinct weights of $U$ give distinct eigenvalues of $s$.  If the tensor product is multiplicity-free for the diagonal torus, the Singer cycle has simple spectrum over $\F_{q^d}$.

We also record a shifted exponent formula for the following, more limited, purpose.  When a module-specific construction changes the Singer eigenvalue data by powers of the $q$--Frobenius map, the exponent vector is obtained by cyclically shifting the base-$q$ digit positions and collecting equal residues modulo $d$.  Under the same bound $K<q-1$, distinct shifted digit vectors are separated by the Singer eigenvalues.  This statement is an eigenvalue-bookkeeping result.  It should not be read as a replacement for Steinberg's tensor product theorem with $p$--Frobenius twists; if one wants to treat general Steinberg twists, the exponent shifts are governed by powers of $p$, and additional arguments are required.

Thus the key input is a number-theoretic observation: under the bound $K<q-1$, bounded base-$q$ digit data are uniquely determined modulo $q^d-1$.  In the untwisted case this yields a uniform replacement for the case-by-case eigenvalue calculations that arise in low-degree tensor examples.  In shifted settings it isolates the precise combinatorial datum that controls the eigenvalue calculation, namely the shifted digit vector.

If, in addition, the relevant weight spaces are multiplicity-free and the combinatorics of the module identify the eigenspaces with those weights or shifted digit vectors, then the eigenspaces of $s$ are one-dimensional.  Building on this, we describe an algorithmic framework for rewriting representations in this class.  The framework consists of:
\begin{itemize}
  \item finding a Singer-type element, or the image of a Singer cycle modulo a central kernel;
  \item labelling the resulting eigenlines by compatible base-$q$ digit vectors;
  \item using the combinatorics of polynomial functors to relate the action on $W$ to the unknown natural action on $V$;
  \item reducing the reconstruction of the natural action to an explicit inversion and normalization problem for the relevant Schur functors.
\end{itemize}

Accordingly, the main unconditional contribution of the paper is the spectral separation and eigenline-labelling mechanism. The reconstruction step is presented as a conditional reduction to a functor-specific inversion problem. We support this framework with illustrative computations, including a full algebraic reconstruction of the natural action from a genuine tensor product representation, rather than claiming a fully uniform rewriting theorem for arbitrary polynomial highest weights. Special-linear and purely projective variants, where one naturally works with Singer subgroups of order $(q^d-1)/(q-1)$ rather than genuine Singer cycles of order $q^d-1$, are left outside the scope of the present paper.

\textbf{The paper is organised as follows.}
In Section~\ref{sec:preliminaries} we collect notation and recall basic facts about polynomial representations of $\GL_d(q)$ and tensor products.
Section~\ref{sec:number-theory} contains the number-theoretic injectivity lemma which underlies our spectral analysis.
Section~\ref{sec:distinct-eigenvalues} establishes the distinct eigenvalue property and the simple spectrum property under multiplicity-freeness.
Section~\ref{sec:algorithm} describes the resulting algorithmic framework and formulates a conditional reconstruction statement.
Section~\ref{sec:sage} presents illustrative \textsf{SageMath} code for core subroutines.
Finally, Section~\ref{sec:computations} reports on computational experiments that verify the base-$q$ injectivity lemma, demonstrate the simple spectrum property for symmetric powers, and explicitly execute the rewriting framework on a strictly multiplicity-free, genuine tensor product of the natural module.

\section{Preliminaries}
\label{sec:preliminaries}

\subsection{Finite fields, the natural module, and Singer cycles}

Throughout, $p$ denotes a prime and $q = p^f$ a prime power, with $f \geq 1$.
We write $\F_q$ for the finite field of order $q$ and $\F_{q^d}$ for its extension of degree $d$.
The multiplicative group of $\F_{q^d}$ is cyclic of order $q^d-1$.

Let $d \geq 2$ and let $V$ be a $d$--dimensional vector space over $\F_q$.
We write $\GL_d(q)$ for the group of invertible linear transformations of $V$, and $\SL_d(q)$ for the subgroup of determinant~1 transformations.
We fix a basis of $V$ and identify $\GL_d(q)$ with the group of invertible $d \times d$ matrices over $\F_q$.
We also write $\PGL_d(q)$ for the projective general linear group $\GL_d(q)/Z(\GL_d(q))$.

\begin{definition}\label{def:singer}
A \emph{Singer cycle} in $\GL_d(q)$ is an element of order $q^d-1$.
Equivalently, after identifying $V$ with $\F_{q^d}$ as an $\F_q$--vector space, it is the linear transformation given by multiplication by a generator of $\F_{q^d}^{\times}$.
In particular, over $\F_{q^d}$ its eigenvalues form a single orbit
\[
  \omega,\ \omega^q,\ \omega^{q^2},\ \dots,\ \omega^{q^{d-1}},
\]
where $\omega$ is a generator of $\F_{q^d}^{\times}$.
\end{definition}

More generally, an element of $\GL_d(q)$ whose characteristic polynomial is irreducible of degree $d$ also has eigenvalues forming a single $q$--Frobenius orbit.
However, the arguments in this paper use a genuine Singer cycle, so that exponents are naturally taken modulo $q^d-1$.

Concretely, let $\omega$ be a generator of $\F_{q^d}^{\times}$.
Then there exists a basis of $V \tensor_{\F_q} \F_{q^d}$ with respect to which a Singer cycle $s$ acts diagonally as
\[
  s \cdot e_i = \ell_i e_i, \quad \ell_i = \omega^{q^{i-1}}, \qquad i = 1,\dots,d.
\]

Throughout the spectral sections of this paper, $H$ denotes a subgroup of $\GL_d(q)$ containing such a genuine Singer cycle.
We emphasize that we do not attempt here to treat the special-linear/projective analogue separately: in $\SL_d(q)$ one naturally encounters Singer subgroups of order $(q^d-1)/(q-1)$ rather than elements of order $q^d-1$, and this requires a different treatment of scalar factors.

\subsection{Polynomial representations and highest weights}

We recall basic facts about polynomial representations of $\GL_d$ over fields of positive characteristic.
Our main reference is Jantzen's monograph: see \cite[Part II]{Jantzen2003}.

Let $\Bbbk=\overline{\F}_q$, and view $\GL_d=\GL(V_{\Bbbk})$ as an algebraic group over $\Bbbk$, where $V_{\Bbbk}=V\tensor_{\F_q}\Bbbk$. Thus, in this subsection $\GL_d$ denotes the algebraic group over $\Bbbk$, while $\GL_d(q)$ from the previous subsection is its group of $\F_q$--rational points.
A rational representation of $\GL_d$ is called \emph{polynomial of degree $k$} if, with respect to some (equivalently, any) basis, the corresponding matrix coefficients are homogeneous polynomial functions of degree $k$ in the matrix entries.
Equivalently, degree-$k$ polynomial representations are the finite-dimensional modules for the Schur algebra $S(d,k)$; see, for example, \cite{Jantzen2003}.

Irreducible rational representations of $\GL_d$ are parametrised by dominant weights $\lambda = (\lambda_1,\dots,\lambda_d)$ with integers $\lambda_1 \geq \cdots \geq \lambda_d$.
When all $\lambda_i$ are non-negative, the corresponding module $L(\lambda)$ is polynomial and has degree $k(\lambda) := \lambda_1 + \cdots + \lambda_d.$ For the restriction to the algebraic subgroup $\SL_d \subset \GL_d$, we may regard $\lambda$ modulo multiples of $(1,1,\dots,1)$, but this plays no role in our arguments.

\begin{definition}
Let $\lambda$ be a dominant weight with non-negative entries.
The \emph{polynomial degree} of $L(\lambda)$ is
\[
  k(\lambda) = \sum_{i=1}^d \lambda_i.
\]
For a finite list $\lambda^{(1)},\dots,\lambda^{(r)}$ of such weights, we set
\[
  K := \sum_{t=1}^r k(\lambda^{(t)}).
\]
\end{definition}

We shall need a simple description of the weights of $L(\lambda)$ when $\lambda$ is polynomial.

Let $T$ denote the diagonal torus of $\GL_d$, consisting of elements of the form $\diag(t_1,\dots,t_d)$ with $t_i \in \Bbbk^\times$.
Let $\varepsilon_i : T \to K^\times$ be the character given by $\varepsilon_i(\diag(t_1,\dots,t_d)) = t_i$.
Every weight $\mu$ of a rational $\GL_d$--module can be expressed as
\[
  \mu = \sum_{i=1}^d b_i(\mu) \varepsilon_i
\]
with $b_i(\mu) \in \mathbb{Z}$.

\begin{proposition}
\label{prop:poly-weights}
Let $L(\lambda)$ be an irreducible polynomial $\GL_d$--module of degree $k = k(\lambda)$, and let $\mu$ be a weight of $L(\lambda)$ with respect to $T$.
Then
\[
  b_i(\mu) \in \mathbb{Z}_{\ge 0} \quad \text{for all } i, \qquad \sum_{i=1}^d b_i(\mu) = k.
\]
\end{proposition}

\begin{proof}
This is standard; see, for example, \cite{Jantzen2003}.
Polynomial representations of $\GL_d$ of degree $k$ are controlled by the Schur algebra $S(d,k)$, and their weights are among the weights occurring in the tensor power $V_{\Bbbk}^{\tensor k}$.
Now a pure tensor
\[
e_{i_1} \tensor \cdots \tensor e_{i_k}
\]
has weight
\[
\varepsilon_{i_1} + \cdots + \varepsilon_{i_k}.
\]
Hence every weight of $V_{\Bbbk}^{\tensor k}$ has the form
\[
\sum_{i=1}^d b_i \varepsilon_i
\qquad\text{with}\qquad
b_i \in \mathbb{Z}_{\ge 0},\quad \sum_{i=1}^d b_i = k.
\]
Therefore the same holds for every weight of any degree-$k$ polynomial $\GL_d$--module, and in particular for $L(\lambda)$.
\end{proof}

In particular, we may regard each weight $\mu$ of $L(\lambda)$ as encoded by a vector $b(\mu) = (b_1(\mu),\dots,b_d(\mu))$ in
\[
  \mathcal{B}_k := \Bigl\{ (b_1,\dots,b_d) \in \mathbb{Z}_{\ge 0}^d \;\Bigm|\; \sum_{i=1}^d b_i = k \Bigr\}.
\]

\begin{definition}[Multiplicity-free polynomial module]
\label{def:multiplicity-free}
Let $L(\lambda)$ be a rational polynomial $\GL_d$--module.
We say that $L(\lambda)$ is \emph{multiplicity-free (for the diagonal torus $T$)} if every weight of $L(\lambda)$ with respect to $T$ occurs with multiplicity~$1$, i.e.\ every weight space is one-dimensional.
More generally, a finite tensor product $W = \bigotimes_{t=1}^r L(\lambda^{(t)})$ is called multiplicity-free if each of its weights (as a $T$--module) occurs with multiplicity~$1$.
\end{definition}

Important examples, at the level of weight-space multiplicities, include the symmetric powers $\Sym^k(V)$ and exterior powers $\wedge^k V$ of the natural module $V$; in these cases each weight is determined uniquely by the multiset of indices and hence occurs with multiplicity~$1$.  No irreducibility assertion for $\Sym^k(V)$ is intended in small characteristic; for example, when $k\ge p$ the symmetric power need not be the irreducible module $L(k\lambda_1)$.

\subsection{\texorpdfstring{$q$--Frobenius shifts}{q-Frobenius shifts} of Singer eigenvalue data}

We shall use one shifted exponent calculation later in the paper.  To avoid a possible ambiguity, we state explicitly what is meant here.  Let $s$ be a Singer cycle with eigenvalues
\[
  \ell_i=\omega^{q^{i-1}},\qquad i=1,\dots,d,
\]
on $V\tensor_{\F_q}\F_{q^d}$.  A $q$--Frobenius shift by $e$ sends this Singer eigenvalue data to
\[
  \ell_i^{q^e}=\omega^{q^{i-1+e}}.
\]
Thus, at the level of exponents modulo $q^d-1$, the shift cyclically permutes the base-$q$ digit positions modulo $d$.

This convention is an eigenvalue-bookkeeping convention.  It is not the same as the full Steinberg tensor product theorem for rational representations, where the relevant algebraic twists are usually $p$--Frobenius twists when $q=p^f$.  A $p^e$--Frobenius twist would produce exponent factors $p^e$, not generally cyclic shifts of base-$q$ digits.  The proposition below records only the $q$--shift calculation needed for Singer eigenvalue labelling.

For this subsection only, the notation $L(\lambda)^{[e]}$ means a copy of the weight data of $L(\lambda)$ whose Singer eigenvalues have been shifted by $q^e$ in the above sense.

\begin{proposition}[Shifted exponent formula for Singer eigenvalue data]
\label{prop:twisted-exponents}
Let $s$ be a Singer cycle in $H$, and write its eigenvalues on $V \tensor_{\F_q} \F_{q^d}$ as
\[
  \ell_i = \omega^{q^{i-1}}, \qquad i=1,\dots,d,
\]
where $\omega$ is a generator of $\F_{q^d}^{\times}$.
Let
\[
  W_{\mathrm{sh}} = \bigotimes_{t=1}^r L(\lambda^{(t)})^{[e_t]},
\]
where each $L(\lambda^{(t)})$ is an irreducible polynomial $\GL_d$--module of degree $k_t$.
For each $t$, let
\[
  \mu^{(t)} = \sum_{i=1}^d b_i^{(t)} \varepsilon_i
\]
be a weight of $L(\lambda^{(t)})$.
Then the corresponding shifted pure tensor has Singer eigenvalue
\[
  \omega^E,
  \qquad
  E \equiv \sum_{t=1}^r \sum_{i=1}^d b_i^{(t)} q^{\,i-1+e_t}
  \pmod{q^d-1}.
\]

For $j=1,\dots,d$, define
\[
  c_j
  =
  \sum_{\substack{1 \le t \le r,\ 1 \le i \le d\\ i-1+e_t \equiv j-1 \; (\mathrm{mod}\ d)}}
  b_i^{(t)}.
\]
Then
\[
  E \equiv \sum_{j=1}^d c_j q^{j-1} \pmod{q^d-1},
\]
and
\[
  0 \le c_j \le K := \sum_{t=1}^r k_t
  \qquad\text{for all } j,
  \qquad
  \sum_{j=1}^d c_j = K.
\]
In particular, if $K < q-1$, then Lemma~\ref{lem:injectivity} applies to the shifted digit vector
\[
  \mathbf{c} = (c_1,\dots,c_d).
\]
\end{proposition}

\begin{proof}
On the weight space of weight $\mu^{(t)}$ in $L(\lambda^{(t)})$, the element $s$ acts by
\[
  \prod_{i=1}^d \ell_i^{\,b_i^{(t)}}
  =
  \omega^{\sum_{i=1}^d b_i^{(t)} q^{i-1}}.
\]
Applying the $q$--Frobenius shift by $e_t$ raises this Singer eigenvalue to the $q^{e_t}$--th power.  Hence the shifted contribution of the $t$--th factor is
\[
  \omega^{\sum_{i=1}^d b_i^{(t)} q^{\,i-1+e_t}}.
\]
Multiplying over $t=1,\dots,r$ gives the first congruence for $E$.

Since $q^d\equiv 1\pmod{q^d-1}$, we may reduce the exponents $i-1+e_t$ modulo $d$ and collect the terms with the same residue class.  This yields
\[
  E \equiv \sum_{j=1}^d c_j q^{j-1} \pmod{q^d-1},
\]
with $c_j$ as above.
Finally, each $b_i^{(t)}$ is a non-negative integer and $\sum_i b_i^{(t)}=k_t$ for each $t$.  Hence each $c_j$ is non-negative, each satisfies $c_j\le \sum_t k_t=K$, and summing over $j$ gives $\sum_j c_j=K$.
\end{proof}

\begin{remark}
Proposition~\ref{prop:twisted-exponents} shows that the shifted Singer eigenvalue of a pure tensor is determined by a shifted digit vector
\[
  \mathbf{c}=(c_1,\dots,c_d)\in \mathcal{B}_K.
\]
Hence Lemma~\ref{lem:injectivity} separates distinct shifted digit vectors whenever $K<q-1$.
What is not automatic is that distinct weights, or distinct tensor factors in a module-specific construction, yield distinct shifted digit vectors.  Any application to genuinely twisted modules must therefore check this combinatorial separation separately.
\end{remark}

Steinberg's tensor product theorem describes irreducible rational representations in terms of $p$--Frobenius twists and $p$--restricted weights; see \cite[Part II, \S3]{Jantzen2003}.  We use it only as structural background and do not use it explicitly in the proofs.

\subsection{The rewriting problem for polynomial tensor modules}

Let $H\le \GL_d(q)$ be a subgroup containing a genuine Singer cycle, acting on its natural module $V$ over $\F_q$.  We are given a subgroup $G\le \GL(W)$, with $W$ an $n$--dimensional vector space over $\F_q$.  In algorithmic applications the representation on $W$ need not be faithful on the scalar centre of $H$.  For instance, $\Sym^2(V)$ kills $-I$ when $q$ is odd.  We therefore formulate the setup in a way that allows a central kernel.

We assume that there is a homomorphism
\[
  \rho:H\longrightarrow \GL(W)
\]
with image $G$ and kernel $Z_0\le Z(H)$.  The faithful case is the special case $Z_0=1$, where $G\cong H$.  We further assume that:
\begin{itemize}
  \item $G$ acts irreducibly on $W$;
  \item over $\Bbbk=\overline{\F}_q$, the $\Bbbk G$--module $W\tensor_{\F_q}\Bbbk$ is identified with the restriction to $H$ of an untwisted tensor product
  \[
    \bigotimes_{t=1}^r L(\lambda^{(t)}),
  \]
  where each $L(\lambda^{(t)})$ is an irreducible polynomial representation of $\GL_d$ of degree $k_t$, and $K:=\sum_{t=1}^r k_t<q-1$.
\end{itemize}

For clarity, the conditional reduction theorem in Section~\ref{sec:algorithm} is formulated in this untwisted setting.  Shifted Singer eigenvalue data are treated only at the level of Proposition~\ref{prop:twisted-exponents}; a reconstruction framework for arbitrary Steinberg-twisted tensor products would require additional module-specific input.

We make the standard algorithmic assumptions that:
\begin{itemize}
  \item we have access to the generating set $X$ of $G$ as matrices in $\GL_n(\F_q)$;
  \item we can multiply matrices and compute in $\F_q$ and in $\F_{q^d}$ when needed;
  \item we can sample nearly uniform random elements of $G$ at cost $\xi$ per element.
\end{itemize}

The \emph{rewriting problem} in this context is:

\medskip\noindent
\textbf{Problem.}
\emph{Construct a projective representation
\[
  \varphi:G\longrightarrow \PGL_d(q)
\]
that is equivalent to the natural projective representation of $H$ on $V$ after quotienting by the central kernel $Z_0$, and such that $\varphi(g)$ can be effectively computed from the matrix of $g$ on $W$.}

\medskip

In the present paper we treat this problem only when the image in $G$ of a genuine Singer cycle of $H$ is available and can be labelled compatibly with the natural Singer eigenvalues.  Because a central kernel may reduce the order of that image, we avoid requiring the matrix in $G$ itself to have order exactly $q^d-1$.  Special-linear and purely projective variants, where one naturally works with Singer subgroups of order $(q^d-1)/(q-1)$ rather than genuine Singer cycles of order $q^d-1$, require a separate treatment and are not pursued here.

\medskip

Our goal is to develop a rewriting framework based on spectral labelling, and to formulate a conditional reduction theorem under the additional hypothesis that $W$ is multiplicity-free (Definition~\ref{def:multiplicity-free}).

\section{A number-theoretic injectivity lemma}
\label{sec:number-theory}

The key technical ingredient in our analysis is the following elementary lemma about writing integers in base $q$.
It asserts that, under a suitable bound on the digits, the map from digit vectors to residues modulo $q^d-1$ is injective.

\begin{lemma}[Base-$q$ injectivity]
\label{lem:injectivity}
Let $q \ge 2$, $d \ge 1$ and let $C$ be an integer with $0 \le C < q-1$.
Consider the set
\[
  \mathcal{B}_C \;=\; \bigl\{ \mathbf{b} = (b_1,\dots,b_d) \in \mathbb{Z}_{\ge 0}^d \bigm| 0 \le b_i \le C \text{ for all } i \bigr\}
\]
and the map
\[
  \Phi : \mathcal{B}_C \longrightarrow \mathbb{Z}/(q^d-1)\mathbb{Z}, \qquad
  \Phi(\mathbf{b}) = \sum_{i=1}^d b_i q^{i-1} \bmod (q^d-1).
\]
Then $\Phi$ is injective.
\end{lemma}

\begin{proof}
It is useful to distinguish the residue class from its standard integer representative.  For
$\mathbf{b}\in\mathcal{B}_C$ put
\[
  \widetilde\Phi(\mathbf{b})=\sum_{i=1}^d b_iq^{i-1}\in\mathbb{Z}.
\]
Then
\[
  0\le \widetilde\Phi(\mathbf{b})
  \le C\sum_{i=1}^d q^{i-1}
  = C\,\frac{q^d-1}{q-1}.
\]
Since $C<q-1$, we have
\[
  C\,\frac{q^d-1}{q-1}<q^d-1.
\]
Thus every $\widetilde\Phi(\mathbf{b})$ lies in the interval $[0,q^d-2]$.

Suppose $\Phi(\mathbf{b})=\Phi(\mathbf{c})$ in $\mathbb{Z}/(q^d-1)\mathbb{Z}$.  Then
\[
  \widetilde\Phi(\mathbf{b})-\widetilde\Phi(\mathbf{c})
\]
is divisible by $q^d-1$.  But both integer representatives lie in $[0,q^d-2]$, so their difference has absolute value strictly less than $q^d-1$.  Hence the difference is zero, and
\[
  \widetilde\Phi(\mathbf{b})=\widetilde\Phi(\mathbf{c})
\]
as integers.

The two sides are base-$q$ expansions
\[
  b_1+b_2q+\cdots+b_dq^{d-1}
  =c_1+c_2q+\cdots+c_dq^{d-1}.
\]
Since $0\le b_i,c_i\le C<q-1<q$, all digits lie in $\{0,\dots,q-1\}$.  Uniqueness of base-$q$ expansion gives $b_i=c_i$ for all $i$, and hence $\mathbf{b}=\mathbf{c}$.
\end{proof}

This lemma will be applied to the weight multiplicities of polynomial modules, which will play the role of the digits $b_i$.

\section{Distinct eigenvalues for tensor products}
\label{sec:distinct-eigenvalues}

We now combine Lemma~\ref{lem:injectivity} with the weight structure of polynomial modules to establish a general distinct-eigenvalue property for Singer cycles on tensor products, and a simple spectrum result under multiplicity-freeness.

\subsection{Eigenvalues of a Singer cycle on a polynomial module}

Let $s$ be a Singer cycle in $H$.
Over $\F_{q^d}$ we may choose a basis of $V \tensor_{\F_q} \F_{q^d}$ such that $s$ acts as
\[
  s \cdot e_i = \ell_i e_i, \qquad \ell_i = \omega^{q^{i-1}}, \quad i = 1,\dots,d,
\]
for some generator $\omega$ of $\F_{q^d}^{\times}$.

Let $L(\lambda)$ be an irreducible polynomial representation of $\GL_d$ of degree $k = k(\lambda)$, realised over $\Bbbk$.
Let $T$ be the diagonal torus as above, and let $\mu$ be a weight of $L(\lambda)$ with weight vector $b(\mu) = (b_1(\mu),\dots,b_d(\mu))$ as in Proposition~\ref{prop:poly-weights}.
Then for $t = \diag(t_1,\dots,t_d) \in T$ the action on a weight vector of weight $\mu$ is
\[
  t \cdot v = \biggl( \prod_{i=1}^d t_i^{\,b_i(\mu)} \biggr) v.
\]

In particular, if we regard $s$ as an element of $T$ over $\Bbbk$, then $s$ acts on this weight space with eigenvalue
\[
  \prod_{i=1}^d \ell_i^{\,b_i(\mu)} = \prod_{i=1}^d \omega^{b_i(\mu) q^{i-1}}
  = \omega^{E(\mu)},
\]
where
\begin{equation}\label{eq:E-mu}
  E(\mu) = \sum_{i=1}^d b_i(\mu) q^{i-1} \in \mathbb{Z}/(q^d-1)\mathbb{Z}.
\end{equation}
By Proposition~\ref{prop:poly-weights} we have $b_i(\mu) \ge 0$ and $\sum_i b_i(\mu) = k$, so $b_i(\mu) \in [0,k]$.

\subsection{Eigenvalues on tensor products}

Let $\lambda^{(1)},\dots,\lambda^{(r)}$ be dominant weights with non-negative entries, and suppose $L(\lambda^{(t)})$ is polynomial of degree $k_t = k(\lambda^{(t)})$.
Consider the tensor product
\[
  W = \bigotimes_{t=1}^r L(\lambda^{(t)}).
\]
As a module for the diagonal torus $T$, the weights of $W$ are sums
\[
  \nu = \mu^{(1)} + \cdots + \mu^{(r)},
\]
where $\mu^{(t)}$ is a weight of $L(\lambda^{(t)})$.
Let $b^{(t)}(\mu^{(t)}) = (b_1^{(t)},\dots,b_d^{(t)})$ denote the corresponding weight vector, so that
\[
  \mu^{(t)} = \sum_{i=1}^d b_i^{(t)} \varepsilon_i.
\]

\begin{definition}
For a weight $\nu$ of $W$ as above, define
\[
  c_i(\nu) = \sum_{t=1}^r b_i^{(t)}, \qquad
  \mathbf{c}(\nu) = (c_1(\nu),\dots,c_d(\nu)).
\]
\end{definition}

Then $c_i(\nu) \ge 0$ and
\[
  \sum_{i=1}^d c_i(\nu) = \sum_{t=1}^r \sum_{i=1}^d b_i^{(t)} = \sum_{t=1}^r k_t = K.
\]
Thus $\mathbf{c}(\nu) \in \mathcal{B}_K$, where $K$ is the total polynomial degree.

The eigenvalue of $s$ on the weight space of $\mu^{(t)}$ in $L(\lambda^{(t)})$ is $\omega^{E(\mu^{(t)})}$ with
\[
  E(\mu^{(t)}) = \sum_{i=1}^d b_i^{(t)} q^{i-1}.
\]
Therefore, the eigenvalue of $s$ on a pure tensor
\[
  v_1 \tensor \cdots \tensor v_r \in W
\]
with $v_t$ in the weight space of $\mu^{(t)}$ is
\[
  \omega^{E(\mu^{(1)})} \cdots \omega^{E(\mu^{(r)})}
  = \omega^{E(\nu)}, \qquad
  E(\nu) = \sum_{t=1}^r E(\mu^{(t)}) = \sum_{i=1}^d c_i(\nu) q^{i-1}.
\]
Hence the eigenvalue of $s$ on the weight space of $\nu$ in $W$ is $\omega^{E(\nu)}$, where $E(\nu)$ depends only on $\mathbf{c}(\nu)$.

\begin{remark}
The untwisted tensor product case is the setting in which the weight $\nu$ itself determines the digit vector $\mathbf{c}(\nu)$ and hence the eigenvalue of a Singer cycle.
For shifted Singer eigenvalue data, Proposition~\ref{prop:twisted-exponents} shows that the relevant invariant is instead the shifted digit vector.
Accordingly, the theorem below is stated only for untwisted tensor products; applications to genuinely twisted modules require a separate check that the relevant shifted digit vectors are distinct.
\end{remark}

\subsection{Distinct eigenvalues for distinct weights}

We can now state and prove the first main spectral result in a form compatible with the base-field module $W$ and its scalar extension to $\overline{\F}_q$.

\begin{theorem}[Distinct eigenvalues for different weights]
\label{thm:distinct-eigenvalues-weak}
Let $H \le \GL_d(q)$ be a subgroup containing a Singer cycle $s$, and let $W$ be an $\F_qH$--module.
Assume that, over
\[
  \Bbbk=\overline{\F}_q,
\]
there is an isomorphism of $\Bbbk H$--modules
\[
  W_{\Bbbk} := W \tensor_{\F_q} \Bbbk
  \;\cong\;
  \bigotimes_{t=1}^r L(\lambda^{(t)}),
\]
where each $L(\lambda^{(t)})$ is an irreducible polynomial $\GL_d$--module of degree $k_t$, and let
\[
  K := \sum_{t=1}^r k_t.
\]
Assume $K < q-1$.

Then for any two distinct weights
\[
  \nu \neq \nu'
\]
of the tensor product
\[
  \bigotimes_{t=1}^r L(\lambda^{(t)})
\]
(viewed as characters of the diagonal torus $T$), the eigenvalues of $s$ on the corresponding transported weight spaces of $W_{\Bbbk}$ are distinct.
\end{theorem}

\begin{proof}
Set
\[
  \Bbbk=\overline{\F}_q
  \qquad\text{and}\qquad
  W_{\Bbbk}:=W \tensor_{\F_q} \Bbbk.
\]
By hypothesis there is a fixed $\Bbbk H$--module isomorphism
\[
  W_{\Bbbk} \;\cong\; \bigotimes_{t=1}^r L(\lambda^{(t)}).
\]
We use this isomorphism to transport the weight-space decomposition of the polynomial tensor product to $W_{\Bbbk}$.  After diagonalising the natural action of the Singer cycle over $\F_{q^d}\subseteq\Bbbk$, the image of $s$ lies in a conjugate of the diagonal torus, and the relevant weights are precisely the weights of the tensor product
\[
  \bigotimes_{t=1}^r L(\lambda^{(t)}).
\]

Let $\nu$ be such a weight.
Write
\[
  \nu = \mu^{(1)} + \cdots + \mu^{(r)},
\]
where $\mu^{(t)}$ is a weight of $L(\lambda^{(t)})$, and write
\[
  \mu^{(t)}=\sum_{i=1}^d b_i^{(t)}\varepsilon_i.
\]
Define
\[
  c_i(\nu):=\sum_{t=1}^r b_i^{(t)},
  \qquad
  \mathbf{c}(\nu):=(c_1(\nu),\dots,c_d(\nu)).
\]
Then
\[
  c_i(\nu)\ge 0
  \qquad\text{and}\qquad
  \sum_{i=1}^d c_i(\nu)=\sum_{t=1}^r k_t = K,
\]
so in particular
\[
  0\le c_i(\nu)\le K
  \qquad\text{for all }i.
\]

By the eigenvalue formula of Section~\ref{sec:distinct-eigenvalues}, the element $s$ acts on the weight space of $\nu$ by the scalar
\[
  \omega^{E(\nu)},
  \qquad
  E(\nu)=\sum_{i=1}^d c_i(\nu)q^{i-1}
  \in \mathbb{Z}/(q^d-1)\mathbb{Z},
\]
where $\omega$ is a generator of $\F_{q^d}^{\times}$.

Now let $\nu\neq \nu'$ be two distinct weights.
Since weights are characters of $T$, their coefficient vectors differ, so
\[
  \mathbf{c}(\nu)\neq \mathbf{c}(\nu').
\]
Because each coordinate of these vectors lies in $[0,K]$ and $K<q-1$, Lemma~\ref{lem:injectivity} shows that the map
\[
  \mathbf{c}\longmapsto \sum_{i=1}^d c_i q^{i-1}\pmod{q^d-1}
\]
is injective on this range.
Hence
\[
  E(\nu)\neq E(\nu')
  \qquad\text{in}\qquad
  \mathbb{Z}/(q^d-1)\mathbb{Z}.
\]
Therefore
\[
  \omega^{E(\nu)}\neq \omega^{E(\nu')},
\]
so the eigenvalues of $s$ on the corresponding weight spaces are distinct.
\end{proof}

\begin{remark}
In all applications in this paper, we are interested in modules of positive total polynomial degree $K>0$. The hypothesis $K<q-1$ then forces $q\ge 3$. Indeed, if $q=2$ then $K<q-1$ implies $K<1$ and hence $K=0$, so there are no non-trivial polynomial tensor products satisfying our standing assumption. Thus Theorems~\ref{thm:distinct-eigenvalues-weak} and~\ref{thm:simple-spectrum} are non-vacuous only for $q \ge 3$.

Moreover, Theorem~\ref{thm:distinct-eigenvalues-weak} is a statement about distinct weights of the tensor product
\[
W_{\Bbbk}\cong \bigotimes_{t=1}^r L(\lambda^{(t)})
\]
in the untwisted setting; it does not address the dimensions of the corresponding weight spaces, which may be larger than one in general. For shifted Singer eigenvalue data, Proposition~\ref{prop:twisted-exponents} shows that the natural invariant controlling the eigenvalue is the shifted digit vector rather than the original weight itself.
\end{remark}

\subsection{Simple spectrum under multiplicity-freeness}

We now state the strengthened result under the additional assumption that the corresponding tensor product over $\overline{\F}_q$ is multiplicity-free.

\begin{theorem}[Simple spectrum under multiplicity-freeness]
\label{thm:simple-spectrum}
Let $H \le \GL_d(q)$ be a subgroup containing a Singer cycle $s$, and let $W$ be an $\F_qH$--module.
Assume that, over
\[
  \Bbbk=\overline{\F}_q,
\]
there is an isomorphism of $\Bbbk H$--modules
\[
  W_{\Bbbk} := W \tensor_{\F_q} \Bbbk
  \;\cong\;
  \bigotimes_{t=1}^r L(\lambda^{(t)}),
\]
where each $L(\lambda^{(t)})$ is an irreducible polynomial $\GL_d$--module of degree $k_t$, and the total polynomial degree
\[
  K := \sum_{t=1}^r k_t
\]
satisfies $K<q-1$.

Assume moreover that this tensor product is multiplicity-free for the diagonal torus $T$ in the sense of Definition~\ref{def:multiplicity-free}.
Then every eigenspace of $s$ on
\[
  W \tensor_{\F_q} \F_{q^d}
\]
is one-dimensional.
Equivalently, $s$ has simple spectrum after scalar extension to the splitting field $\F_{q^d}$.
\end{theorem}

\begin{proof}
Set
\[
  \Bbbk=\overline{\F}_q
  \qquad\text{and}\qquad
  W_{\Bbbk}:=W \tensor_{\F_q} \Bbbk.
\]
By hypothesis there is a fixed $\Bbbk H$--module isomorphism
\[
  W_{\Bbbk} \;\cong\; \bigotimes_{t=1}^r L(\lambda^{(t)}).
\]
We transport the weight-space decomposition of this polynomial tensor product to $W_{\Bbbk}$.  After diagonalising the natural action of $s$ over $\F_{q^d}\subseteq\Bbbk$, this gives a decomposition
\[
  W_{\Bbbk} \;=\; \bigoplus_{\nu} (W_{\Bbbk})_\nu,
\]
where the summands are the transported weight spaces.

Each weight space $(W_{\Bbbk})_\nu$ is stable under $s$, and $s$ acts on $(W_{\Bbbk})_\nu$ by the scalar $\nu(s)$.
By multiplicity-freeness, each weight space $(W_{\Bbbk})_\nu$ is one-dimensional.
Moreover, since the hypotheses of Theorem~\ref{thm:distinct-eigenvalues-weak} apply to the tensor product
\[
  \bigotimes_{t=1}^r L(\lambda^{(t)}),
\]
distinct weights $\nu \neq \nu'$ give distinct eigenvalues
\[
  \nu(s) \neq \nu'(s).
\]
Therefore each eigenspace of $s$ on $W_{\Bbbk}$ is exactly one weight space $(W_{\Bbbk})_\nu$, and is thus one-dimensional.

By the eigenvalue formula of Section~\ref{sec:distinct-eigenvalues}, every eigenvalue of $s$ on $W_{\Bbbk}$ is of the form $\omega^E$ for some generator $\omega \in \F_{q^d}^{\times}$.
Hence all eigenvalues of $s$ lie in $\F_{q^d}$.

Now set
\[
  W_{q^d}:=W \tensor_{\F_q} \F_{q^d}.
\]
For any eigenvalue $\lambda \in \F_{q^d}$ of $s$, the $\lambda$--eigenspace on $W_{q^d}$ is
\[
  \ker_{W_{q^d}}(s-\lambda I).
\]
Since scalar extension from $\F_{q^d}$ to $\Bbbk$ is exact, we have
\[
  \ker_{W_{q^d}}(s-\lambda I) \tensor_{\F_{q^d}} \Bbbk
  \;\cong\;
  \ker_{W_{\Bbbk}}(s-\lambda I).
\]
The right-hand side is the $\lambda$--eigenspace of $s$ on $W_{\Bbbk}$, which has already been shown to be one-dimensional.
Therefore
\[
  \dim_{\F_{q^d}} \ker_{W_{q^d}}(s-\lambda I)=1.
\]

Hence every eigenspace of $s$ on
\[
  W \tensor_{\F_q} \F_{q^d}
\]
is one-dimensional.
Equivalently, $s$ has simple spectrum over the splitting field $\F_{q^d}$.
\end{proof}

\begin{remark}[Relation with the modules of G\"ul--Ankaral{\i}o\u{g}lu]
The modules considered by G\"ul and Ankaral{\i}o\u{g}lu in~\cite{GulAnk2016} are twisted tensor products of modules with highest weights among
\[
  \lambda_1, \lambda_2, \lambda_{d-2}, \lambda_{d-1}, 2\lambda_1, 2\lambda_{d-1}.
\]
Under the polynomial $\GL_d$ convention used in this paper, the weights $\lambda_1$, $\lambda_2$ and $2\lambda_1$ correspond to $V$, $\wedge^2V$ and $\Sym^2(V)$, of polynomial degrees $1$, $2$ and $2$, respectively.  By contrast, $\lambda_{d-1}$ and $\lambda_{d-2}$ are dual-side weights from the $\SL_d$ viewpoint.  Their polynomial $\GL_d$ realisations as exterior powers of $V$ have degrees $d-1$ and $d-2$, and $2\lambda_{d-1}$ is not a degree-two polynomial $\GL_d$ module under this convention.  Thus the present polynomial $\GL_d$ framework does not directly cover all of the dual or projectively normalised cases treated in~\cite{GulAnk2016}.

Theorem~\ref{thm:simple-spectrum} also does not by itself recover all of the cases treated in~\cite{GulAnk2016}, because multiplicity-freeness of individual factors does not automatically imply multiplicity-freeness of the full tensor product, and because shifted or Steinberg-twisted factors require additional combinatorial separation checks.  What the present argument provides is a uniform explanation for the eigenvalue-separation mechanism once the relevant polynomial or shifted digit data are known to be separated.  In this sense, our results complement the case-by-case calculations of~\cite{GulAnk2016}, rather than replacing their module-specific analysis.
\end{remark}

\section{An algorithmic framework for rewriting}
\label{sec:algorithm}

In this section we explain how the spectral results of Section~\ref{sec:distinct-eigenvalues} lead to a rewriting framework for representations in the class covered by Theorem~\ref{thm:simple-spectrum}.  The key point is that a Singer-type element with simple spectrum gives canonically labelled \emph{eigenlines}.  Eigenvectors themselves are still determined only up to non-zero scalar multiples, and those scalar choices form part of the functor-specific normalization problem.

The discussion below should therefore be read as a conditional reduction rather than as a complete uniform rewriting algorithm for arbitrary polynomial highest weights.

\subsection{Outline of the framework}

Let $G\le \GL(W)$ be as in the problem statement of Section~\ref{sec:preliminaries}.  Thus $G$ is the image of a group $H\le \GL_d(q)$ containing a genuine Singer cycle, possibly modulo a central kernel, and
\[
  W\tensor_{\F_q}\overline{\F}_q
\]
is identified with an untwisted polynomial tensor product which is multiplicity-free for the diagonal torus and has total degree  $K<q-1$.

\medskip\noindent
\textbf{Step~1: Obtain a Singer-type element.}

The starting point is an element $s_G\in G$ which is the image of a genuine Singer cycle $\widetilde{s}\in H$ under the representation $\rho:H\to G$.  Because $\rho$ may have a central kernel, the order of $s_G$ can be a proper divisor of $q^d-1$.  Thus, in this framework, one should search for or certify the image of a Singer cycle rather than insist that the observed matrix on $W$ itself has order exactly $q^d-1$.

In the families treated by earlier rewriting algorithms, such elements can often be obtained by random search, using standard nearly uniform random element generators together with order tests, primitive prime divisors, and irreducible degree-$d$ factors; see, for example, \cite{GulAnk2016,MagaardOBrienSeress2008}.  For the purposes of the present reduction, this search step is treated as an external Las Vegas subroutine.

Once such an element is available, Theorem~\ref{thm:simple-spectrum} gives simple spectrum for its action on
\[
  W\tensor_{\F_q}\F_{q^d}
\]
under the stated multiplicity-free hypotheses.

\medskip\noindent
\textbf{Step~2: Label eigenlines via compatible base-$q$ data.}

Let $\lambda_1,\dots,\lambda_n$ be the eigenvalues of $s_G$ on $W\tensor_{\F_q}\F_{q^d}$.  In the theoretical model, these eigenvalues have the form
\[
  \lambda_j=\omega^{E_j},
\]
where $\omega$ is a Singer eigenvalue of $\widetilde{s}$ on the natural module $V$.

A crucial normalization issue is that the input representation on $W$ does not by itself determine this primitive Singer eigenvalue $\omega$.  If one chooses an arbitrary generator $\omega'=\omega^m$ of $\F_{q^d}^{\times}$, then the discrete logarithms are multiplied by $m^{-1}$ modulo $q^d-1$, and the resulting base-$q$ digits need not lie in the expected bounded set.  Therefore a compatible choice of $\omega$, or an equivalent Singer-labelling oracle, is an explicit hypothesis of the reconstruction theorem below.

With such a compatible normalization fixed, write
\[
  E_j=c_1^{(j)}+c_2^{(j)}q+\cdots+c_d^{(j)}q^{d-1},
\]
where $0\le E_j<q^d-1$.  By Proposition~\ref{prop:poly-weights}, Theorem~\ref{thm:distinct-eigenvalues-weak}, and Lemma~\ref{lem:injectivity}, the digits satisfy
\[
  0\le c_i^{(j)}\le K,
  \qquad
  \sum_{i=1}^d c_i^{(j)}=K,
\]
and the eigenline corresponding to $\lambda_j$ receives the unique label
\[
  \mathbf{c}^{(j)}=(c_1^{(j)},\dots,c_d^{(j)})\in\mathcal{B}_K.
\]
For shifted Singer eigenvalue data, Proposition~\ref{prop:twisted-exponents} gives the analogous labelling by shifted digit vectors, provided that the shifted combinatorial data are known to be separated.

\medskip\noindent
\textbf{Step~3: Choose and normalize eigenvectors.}

For each eigenvalue $\lambda_j$, compute the one-dimensional eigenspace
\[
  L_j=\ker(s_G-\lambda_jI)\subseteq W\tensor_{\F_q}\F_{q^d}.
\]
The spectral theorem supplies a labelled set of eigenlines
\[
  L_j\longleftrightarrow \mathbf{c}^{(j)}.
\]
Choosing a vector $f_j\in L_j$ produces an eigenbasis, but this basis is not canonical: each $f_j$ may be multiplied by an arbitrary element of $\F_{q^d}^{\times}$.  In concrete rewriting algorithms these scalars must be fixed by a functor-specific normalization, or else absorbed into the basis-identification step.

\medskip\noindent
\textbf{Step~4: Reduce to a functor-specific basis-identification and inversion problem.}

Let $g\in G$, and let $M_{\mathrm{eig}}(g)$ denote the matrix of $g$ with respect to a chosen labelled eigenbasis.  This basis need not be the standard functorial basis in which the action of $g$ is expressed by known polynomial functions in the entries of an unknown natural matrix $A(g)$ on $V$.

Thus an intermediate problem remains.

\medskip\noindent
\emph{Basis-identification and scaling problem.}
\emph{Use the eigenline labels $\mathbf{c}^{(j)}$, the scalar choices in the eigenbasis, and the combinatorics of the relevant polynomial functors to identify the labelled eigendata with a weight basis, tableau basis, or other functorial basis in which the induced action is explicitly known.}

\medskip

Once this identification is made, the matrix of $g$ in the resulting functorial basis, denoted $M_W(g)$, is obtained from $M_{\mathrm{eig}}(g)$ by a known change of basis.  The entries of $M_W(g)$ are then polynomial expressions in the entries of the unknown matrix $A(g)$ on the natural module.

For example, if $W=V\tensor V$, then
\[
  M_W(g)=A(g)\tensor A(g).
\]
For symmetric powers, exterior powers, and Schur functors, the analogous induced matrices are obtained from explicit polynomial formulas in the entries of $A(g)$.

This leads to the reconstruction problem.

\medskip\noindent
\emph{Functor-specific inversion problem.}
\emph{Given $M_W(g)$ and the basis-identification data, recover a matrix $A(g)$ on $V$, up to the scalar and field-automorphism ambiguities inherent in the functor, such that the prescribed polynomial functor applied to $A(g)$ yields $M_W(g)$.}

\medskip

For specific functors, such as symmetric powers, exterior powers, and the low-degree modules treated in the literature, this inversion can often be carried out explicitly by selecting entries corresponding to simple monomials and solving the resulting equations.  In the general setting of arbitrary polynomial highest weights, however, this is a module-specific problem.

\medskip\noindent
\textbf{Step~5: Assemble and descend the projective representation.}

Assume that the functor-specific reconstruction can be performed consistently for the generators $x\in X$.  For each generator one obtains a projective class
\[
  [A(x)]\in \PGL_d(\F_{q^d}).
\]
Additional descent and normalization checks are then needed to ensure that these classes lie in a conjugate of $\PGL_d(q)$ and that they define a homomorphism
\[
  \varphi:G\longrightarrow \PGL_d(q).
\]
This descent is included as an explicit hypothesis in the theorem below.

\subsection{A conditional reduction theorem}

We now formulate the preceding framework as a precise conditional reduction.  The spectral labelling mechanism is supplied by the results of Sections~\ref{sec:number-theory} and~\ref{sec:distinct-eigenvalues}.  The Singer search, compatible choice of $\omega$, basis scaling, functor inversion, and descent to $\PGL_d(q)$ are external inputs.

Throughout this section, we treat arithmetic in $\F_{q^d}$, the generation of nearly uniform random elements of $G$, and discrete logarithm computations in $\F_{q^d}^{\times}$ as basic subroutines.  All complexity bounds are measured in terms of the dimension $n=\dim W$, the parameters $d$, $\log q$, and $K$, the cost~$\xi$ of generating random elements of $G$, the cost of field operations in $\F_{q^d}$, and the cost~$\delta_{q^d}$ of discrete logarithm computations when such logarithms are used.

\begin{theorem}[Conditional reduction to functor-specific reconstruction]
\label{thm:rewriting}
Let $q=p^f$ be a prime power, and let $H\le \GL_d(q)$ be a subgroup containing a genuine Singer cycle.  Let $V$ be the natural $d$--dimensional $\F_qH$--module.  Suppose that $G\le \GL(W)$ is the image of a homomorphism
\[
  \rho:H\longrightarrow \GL(W)
\]
with central kernel $Z_0\le Z(H)$, and that $G$ acts irreducibly on the $n$--dimensional $\F_q$--vector space $W$.

Assume that, over $\overline{\F}_q$, the module $W\tensor_{\F_q}\overline{\F}_q$ is identified with the restriction to $H$ of an untwisted tensor product
\[
  \bigotimes_{t=1}^r L(\lambda^{(t)}),
\]
where each $L(\lambda^{(t)})$ is an irreducible polynomial representation of $\GL_d$ of degree $k_t$, the total degree
\[
  K:=\sum_{t=1}^r k_t
\]
satisfies $K<q-1$, and the tensor product is multiplicity-free for the diagonal torus.  Assume moreover that the polynomial functors defining the factors are known explicitly.

Suppose that an element $s_G\in G$ is given which is the image under $\rho$ of a genuine Singer cycle $\widetilde{s}\in H$, and that a compatible primitive Singer eigenvalue $\omega\in\F_{q^d}^{\times}$ for $\widetilde{s}$ is given, or equivalently that a Singer-labelling oracle is available.  Then the spectral results of Sections~\ref{sec:number-theory} and~\ref{sec:distinct-eigenvalues} yield a deterministic procedure which:
\begin{itemize}
  \item computes the eigenvalues and one-dimensional eigenspaces of $s_G$ on $W\tensor_{\F_q}\F_{q^d}$;
  \item labels the resulting eigenlines by vectors in $\mathcal{B}_K$ via compatible base-$q$ expansions.
\end{itemize}

Assume in addition that:
\begin{itemize}
  \item an image of a genuine Singer cycle can be found by a Las Vegas random search in expected polynomial time;
  \item the compatible Singer eigenvalue normalization, or an equivalent labelling oracle, is available;
  \item for the class of polynomial functors under consideration, the basis-identification, eigenbasis-scaling, and inversion problems of Step~4 can be solved uniformly in polynomial time from the labelled matrices attached to the generators $x\in X$;
  \item the resulting projective matrices descend from $\PGL_d(\F_{q^d})$ to $\PGL_d(q)$ and are compatible on the generators, in the sense that they define a homomorphism
  \[
    \varphi:G\longrightarrow \PGL_d(q).
  \]
\end{itemize}
Then, for any prescribed $\varepsilon\in(0,1)$, this gives a Las Vegas reduction which, with probability at least $1-\varepsilon$, constructs a projective representation
\[
  \varphi:G\longrightarrow \PGL_d(q)
\]
equivalent to the natural projective representation of $H$ on $V$ modulo the central kernel $Z_0$.

The expected running time is polynomial in
\[
  n,\ d,\ \log q,\ K,\ \log(\varepsilon^{-1}),
\]
and in the costs of random element generation, field arithmetic in $\F_{q^d}$, compatible labelling or discrete logarithm computations when used, and the functor-specific basis-identification, scaling, inversion, and descent subroutines.
\end{theorem}

\begin{proof}
Let $s_G=\rho(\widetilde{s})$, where $\widetilde{s}\in H$ is a genuine Singer cycle.  By Theorems~\ref{thm:distinct-eigenvalues-weak} and~\ref{thm:simple-spectrum}, the eigenspaces of $s_G$ on
\[
  W\tensor_{\F_q}\F_{q^d}
\]
are one-dimensional under the stated polynomial-degree and multiplicity-free hypotheses.  The compatible Singer eigenvalue $\omega$ identifies each eigenvalue with an exponent modulo $q^d-1$.  Lemma~\ref{lem:injectivity} then implies that each such exponent determines a unique bounded base-$q$ digit vector in $\mathcal{B}_K$.  This proves the deterministic eigenline-labelling part.

The remaining assertions follow by composition of the assumed subroutines.  The Las Vegas search supplies a suitable Singer-type element with the prescribed success probability.  The basis-identification and scaling routine converts the labelled eigenline data into the functorial coordinates needed for the chosen polynomial functors.  The inversion routine recovers projective natural matrices for the generators.  Finally, the compatibility and descent assumptions ensure that these projective matrices define a homomorphism
\[
  \varphi:G\longrightarrow \PGL_d(q)
\]
equivalent to the natural projective representation of $H$ on $V$ after quotienting by the central kernel.

The stated running time is obtained by adding the expected running times of the Singer search, the deterministic spectral labelling, and the assumed polynomial-time functor-specific routines.  Repetition of the Las Vegas search gives success probability at least $1-\varepsilon$ with the usual logarithmic dependence on $\varepsilon^{-1}$.
\end{proof}

\begin{remark}[Comparison with existing algorithms and limitations]
\label{rem:comparison}
From the perspective of matrix group recognition, the spectral results of this paper provide a module-specific labelling mechanism that can feed into rewriting procedures for suitable classes of polynomial tensor modules.

At the level of eigenvalue analysis, our results replace certain case-by-case calculations by a uniform bounded-digit argument based on Lemma~\ref{lem:injectivity}.  For shifted Singer eigenvalue data, Proposition~\ref{prop:twisted-exponents} gives the corresponding shifted digit calculation.  This isolates a clean spectral mechanism that applies across bounded-degree polynomial tensor constructions subject to the condition $K<q-1$.

What remains functor-specific is the basis-identification, eigenbasis-scaling, inversion, and descent step that reconstructs the natural action from the induced polynomial action.  Accordingly, Theorem~\ref{thm:rewriting} is a reduction statement rather than a complete uniform rewriting theorem for arbitrary polynomial highest weights.

This should be contrasted with the work of Corr~\cite{Corr2015} and of Magaard--O'Brien--Seress~\cite{MagaardOBrienSeress2008}, where the relevant reconstruction steps are carried out explicitly for the classes under consideration.  Likewise, the constructive recognition algorithms of Brooksbank~\cite{Brooksbank2003} and of Dietrich--Leedham-Green--O'Brien~\cite{DietrichLGO2015} operate at a different stage of the recognition pipeline: once a natural or projective-natural copy has been obtained, those algorithms can be used to complete the recognition process.

The hypotheses in Theorem~\ref{thm:rewriting} are deliberately explicit.  The bound $K<q-1$ excludes certain small fields; multiplicity-freeness excludes modules with repeated weight spaces; central kernels can reduce the order of Singer images; an arbitrary generator of $\F_{q^d}^{\times}$ does not provide a valid digit labelling unless it is compatible with the natural Singer eigenvalue; and special-linear or purely projective Singer subgroups require a separate treatment.
\end{remark}

\section{Illustrative \textsf{SageMath} code}
\label{sec:sage}

In this section we present some \textsf{SageMath} code fragments that illustrate core components of the algorithm:
computing eigenvalues of a Singer cycle, extracting base-$q$ expansions, and checking the injectivity property of Lemma~\ref{lem:injectivity} for a given module.

\subsection{\texorpdfstring{Base-$q$ expansion}{Base-q expansion} and injectivity test}

The following function takes an exponent $E$ and returns its base-$q$ expansion in $d$ digits.
We assume $0 \le E < q^d$.

\begin{lstlisting}
def base_q_expansion(E, q, d):
    """
    Return the base-q expansion of integer E as a list of length d:
        E = sum_{i=0}^{d-1} c[i]*q**i,  0 <= c[i] < q.
    """
    coeffs = []
    for _ in range(d):
        coeffs.append(E % q)
        E //= q
    return coeffs  # c[0], ..., c[d-1]
\end{lstlisting}

We can use this to verify the injectivity of the map $\Phi$ for given parameters $(q,d,C)$:

\begin{lstlisting}
def check_injectivity(q, d, C, verbose=True):
    r"""
    Check injectivity of the map

        Phi : B_C -> Z/(q^d - 1)Z,
        Phi(b_1,...,b_d) = sum_{i=0}^{d-1} b_{i+1} q^i  mod (q^d - 1),

    where
        B_C = { (b_1,...,b_d) in Z_{\ge 0}^d | 0 <= b_i <= C }.

    This is an exhaustive search, so exponential in d.
    Intended only for small d and C.
    """
    from itertools import product

    modulus = q**d - 1
    seen = {}

    for b in product(range(C + 1), repeat=d):
        E = sum(b[i] * (q**i) for i in range(d)) % modulus
        if E in seen and seen[E] != b:
            if verbose:
                print("Collision found!")
                print("  b =", b, "and", "c =", seen[E], "both map to", E)
            return False
        seen[E] = b

    if verbose:
        print(f"Phi is injective on B_{C} for (q,d,C) = ({q},{d},{C}).")
        print(f"Checked {len(seen)} vectors.")
    return True
\end{lstlisting}

For small values of $d$ and $C$ one can experimentally confirm Lemma~\ref{lem:injectivity}.

\medskip
\noindent
\textbf{Note.} The implementation above uses \texttt{itertools.product} which is suitable only for small parameters. For the large-scale experiments in Section~\ref{sec:experiments-model} below (e.g., $q=2^{16}, d=10$), our supplementary script employs an optimized recursive generator \texttt{weight\_patterns\_sumK} and performs arithmetic purely on integer exponents modulo $q^d-1$, avoiding the expensive construction of the extension field.

\subsection{Eigenvalues of a Singer cycle on a tensor power}

As a simple model, we consider the action on $V^{\tensor K}$.
In practice, $W$ is a submodule or quotient of $V^{\tensor K}$ corresponding to a Schur functor, but $V^{\tensor K}$ is easier to work with.

\begin{lstlisting}
def singer_eigenvalues_tensor_power(q, d, K):
    """
    Compute eigenvalues of a Singer cycle on V^{\otimes K},
    where V is the natural d-dimensional module over GF(q).

    Returns a dictionary mapping base-q exponent vectors c in B_K
    (with sum(c) = K) to the corresponding eigenvalue in GF(q^d).

    Note: This function verifies the injectivity of the map
        c -> omega^{E(c)} with E(c) = sum c_i q^i,
    for distinct weight patterns c in V^{\otimes K}.
    It does NOT detect possible weight multiplicities in irreducible
    submodules or quotients: in such modules several linearly
    independent vectors may share the same weight pattern c and
    hence the same eigenvalue.
    """
    # Finite fields
    Fq = GF(q)
    Fqd = GF(q**d, 'a')
    a = Fqd.gen()

    # Choose a generator omega of F_{q^d}^*
    omega = Fqd.multiplicative_generator()

    # Dictionary from c-vector to eigenvalue
    eig = {}

    # Enumerate all c in B_K
    from itertools import product

    for c in product(range(K+1), repeat=d):
        if sum(c) != K:
            continue
        # exponent E(c) = sum c_i q^i
        E = sum(c[i] * (q**i) for i in range(d))
        eig_val = omega**E
        eig[c] = eig_val

    return eig
\end{lstlisting}

\medskip
\noindent
\textbf{Implementation note.}
In an actual implementation one should not assume that the default field generator is primitive.  Accordingly, when a generator of $\F_{q^d}^{\times}$ is required, the code should use \texttt{Fqd.multiplicative\_generator()}.  If an explicit Singer matrix is constructed from a chosen primitive element $\omega$, its matrix entries must be the coordinates of multiplication by $\omega$ with respect to the chosen basis $1,\omega,\dots,\omega^{d-1}$.  One should not use \texttt{y.polynomial()} for this purpose unless the field generator used by Sage is known to be the same element as $\omega$.

By inspecting the dictionary returned by \texttt{singer\_eigenvalues\_tensor\_power}, one can verify that different $\mathbf{c}$ give distinct eigenvalues when $K < q-1$, in agreement with Lemma~\ref{lem:injectivity}.

\subsection{Recovering \texorpdfstring{base-$q$}{base-q} labels from eigenvalues}

Suppose we know an eigenvalue $\lambda\in\F_{q^d}$ and we have fixed a generator $\omega$ of $\F_{q^d}^{\times}$.  For the labels to have the interpretation used in Section~\ref{sec:algorithm}, this generator must be compatible with the Singer eigenvalue on the natural module, or else supplied by a labelling oracle.  With that normalization in place, we can compute the exponent $E$ such that $\lambda=\omega^E$ and then extract its base-$q$ expansion.  The following function performs this task.

\begin{lstlisting}
def exponent_and_digits(lam, omega, q, d):
    """
    Given an eigenvalue lam = omega^E in GF(q^d)^*,
    return the exponent E (0 <= E < q^d-1) and its base-q digits.

    The generator omega must be compatible with the Singer eigenvalue
    on the natural module.  An arbitrary primitive generator can give
    digit vectors with no representation-theoretic meaning.
    """
    # Discrete log: find E such that omega^E = lam
    E = lam.log(omega)           # Sage's discrete log
    E = ZZ(E) % (q**d - 1)       # normalise exponent modulo q^d - 1
    digits = base_q_expansion(E, q, d)
    return E, digits
\end{lstlisting}

This function provides the labels $\mathbf{c}^{(j)}$ used in Step~2 of the algorithm only after the compatible normalization of $\omega$ has been fixed.

Besides these basic routines, the revised supplementary script \texttt{singer\_sym\_check.sage} contains helper functions for constructing an explicit Singer matrix in $\GL_d(q)$, computing the induced action on $\Sym^k(V)$, running end-to-end eigenvalue checks, and performing the toy reconstruction experiment from $\Sym^2(A)$ described in Section~\ref{sec:computations} below.  The implementation should follow the coordinate convention for Singer matrices stated above and the symmetric-square formula in Section~\ref{sec:computations} exactly; in particular, the middle coefficient in the second column of $\Sym^2(A)$ is $ad+bc$, not $2(ad+bc)$.  For readability, we omit these longer code fragments here.

\section{Computational experiments}
\label{sec:computations}

In this section, we present some small computational experiments which serve as a sanity check for the number-theoretic Lemma~\ref{lem:injectivity} and for the spectral behaviour of Singer cycles on symmetric powers of the natural module, as well as a toy model for the reconstruction step in the rewriting algorithm.
The computations were performed in \textsf{SageMath}.  The supplementary script accompanying the revised manuscript is intended as a collection of sanity checks for the key mechanisms in the paper and should be kept synchronized with the formulas and coordinate conventions stated here.  The matrix reconstruction routines are designed only for small examples that illustrate correctness; the eigenvalue labelling and injectivity checks are implemented by integer exponent arithmetic and can also be tested on moderately large parameter values.

\subsection{Verification of the \texorpdfstring{base-$q$}{base-q} injectivity lemma}

Recall that Lemma~\ref{lem:injectivity} asserts that, for $C < q-1$, the map
\[
  \Phi : \mathcal{B}_C \to \mathbb{Z}/(q^d-1)\mathbb{Z}, \qquad
  \Phi(\mathbf{b}) = \sum_{i=1}^d b_i q^{i-1} \bmod (q^d-1),
\]
is injective, where
\[
  \mathcal{B}_C = \{ (b_1,\dots,b_d) \in \mathbb{Z}_{\geq 0}^d \mid 0 \leq b_i \leq C \text{ for all } i\}.
\]

For fixed parameters $(q,d,C)$, the script exhaustively enumerates $\mathbf{b} \in \mathcal{B}_C$, computes $\Phi(\mathbf{b})$ and checks for collisions.
Since this is exponential in $d$, we only run it for small values.
For example, with
\[
  q = 7,\quad d = 3,\quad C = 3,
\]
we have $C < q-1$ and $|\mathcal{B}_3| = 4^3 = 64$.
The program confirms that
\[
  \Phi : \mathcal{B}_3 \longrightarrow \mathbb{Z}/(7^3-1)\mathbb{Z}
\]
is injective on all $64$ vectors and reports no collisions.
In particular, a typical run outputs
\[
  \texttt{Phi is injective on B\_3 for (q,d,C) = (7,3,3).}
\]
followed by
\[
  \texttt{Checked 64 vectors.}
\]
Similar calculations for various small triples $(q,d,C)$ with $C < q-1$ consistently support Lemma~\ref{lem:injectivity}.
Of course, these experiments do not constitute a proof, but they provide additional confidence that the base-$q$ injectivity phenomenon behaves as predicted in all small cases.

\subsection{Model eigenvalues on tensor powers}
\label{sec:experiments-model}

Before turning to genuine matrix representations, we also implemented the abstract eigenvalue model described in Section~\ref{sec:distinct-eigenvalues}.
Given parameters $(q,d,K)$ with $K < q-1$, we consider
\[
  \mathcal{B}_K' = \{ \mathbf{c} = (c_1,\dots,c_d) \in \mathbb{Z}_{\geq 0}^d \mid \sum_i c_i = K \},
\]
and define, for each $\mathbf{c} \in \mathcal{B}_K'$,
\[
  E(\mathbf{c}) = \sum_{i=1}^d c_i q^{i-1}, \qquad \lambda_{\mathbf{c}} = \omega^{E(\mathbf{c})} \in \F_{q^d}^{\times},
\]
where $\omega$ is a fixed generator of $\F_{q^d}^{\times}$.
This matches the expression for the eigenvalues of a Singer cycle on $V^{\otimes K}$ (and on polynomial submodules) given in~\eqref{eq:E-mu}.

For $(q,d,K) = (7,3,3)$ the set $\mathcal{B}_3'$ has size $|\mathcal{B}_3'| = \binom{3+3-1}{3-1} = 10$, corresponding to the $10$ ways of writing $3$ as an ordered sum of $3$ non-negative integers.
The script computes $\lambda_{\mathbf{c}}$ for all $\mathbf{c} \in \mathcal{B}_3'$ and verifies that these $10$ eigenvalues are pairwise distinct.
This is reported in the console as
\begin{verbatim}
Number of weight patterns c with sum(c) = 3: 10
\end{verbatim}
followed by
\begin{verbatim}
Distinct weight patterns c give distinct eigenvalues (as expected).
\end{verbatim}

Moreover, for a sample pattern, say $\mathbf{c} = (0,0,3)$, the program obtains an eigenvalue
\[
  \lambda_{\mathbf{c}} = \omega^{147}, \qquad
  147 = 0 + 0 \cdot 7 + 3 \cdot 7^2,
\]
and recovers the base-$7$ digits of $147$ as $(0,0,3)$.
In the actual output this appears as
\begin{verbatim}
Example weight pattern c = (0, 0, 3),
Exponent E = 147,    Base-q digits = [0, 0, 3].
\end{verbatim}
This exactly illustrates the mechanism of Lemma~\ref{lem:injectivity}: the bound $K < q-1$ ensures that the digits $c_i$ can be reconstructed uniquely from the exponent $E(\mathbf{c})$ modulo $q^d-1$.

\medskip

As a more demanding sanity check closer to the intended applications, we also ran a purely ``exponent--model'' variant of this experiment for larger parameters. Specifically, we considered
\[
  (q,d,K) \;=\; (2^{16},\,10,\,K)
  \qquad\text{with}\qquad K \in \{2,3,4\}.
\]
In this setting we do not construct $\F_{q^d}$ or a Singer matrix explicitly, but work only with the exponents $E(\mathbf{c})$ in $\mathbb{Z}/(q^d-1)\mathbb{Z}$.
For each $K$ we enumerate all digit vectors $\mathbf{c} \in \mathcal{B}_K$ of length $d$ with $\sum_i c_i = K$ and verify that the map
\[
  \mathbf{c} \;\longmapsto\; \sum_{i=1}^d c_i q^{i-1} \bmod (q^d - 1)
\]
is injective.
For $K = 2,3,4$ we have
\[
  |\mathcal{B}_K'| \;=\; \binom{d+K-1}{K} \;=\; 55,\ 220,\ 715,\ \ \mbox{respectively},
\]
and in each case the program reports that all $|\mathcal{B}_K'|$ exponents are distinct and that no collisions occur.
This confirms the base-$q$ injectivity lemma in a regime where the corresponding symmetric-power dimension reaches
\[
  \dim \Sym^4(V) \;=\; \binom{10+4-1}{4} \;=\; 715.
\]
Here $\Sym^4(V)$ is used only as a polynomial-module dimension benchmark; since the characteristic is $2$, no irreducibility assertion is being made.  The computation avoids any explicit matrix arithmetic over the large field $\F_{2^{160}}$.

Finally, we demonstrate the feasibility of our labeling strategy for these parameters.
In our implementation, the lookup table for $K=4$ is generated and the corresponding injectivity check modulo $2^{160}-1$ is completed essentially instantaneously on standard hardware.
The precise runtime depends on the machine and the arithmetic backend, so we record this only as an indicative benchmark.
The main point is that, in this exponent model, working directly with exponents and lookup tables is much cheaper than explicit matrix or field computations; this does not by itself constitute a large-field rewriting experiment.

\subsection{Real Singer matrices and symmetric powers}

To test the full representation-theoretic picture, we next worked with genuine matrices in $\GL_d(q)$ and their induced action on symmetric powers of the natural module.

Fix $(q,d) = (7,3)$ and let $\F_{7^3}$ be the field of order $7^3$.
Let $\omega$ be a generator of $\F_{7^3}^{\times}$ and consider $\F_{7^3}$ as a $3$--dimensional vector space over $\F_7$ with basis $\{1,\omega,\omega^2\}$.
Multiplication by $\omega$ defines a linear operator $m_\omega$ on $\F_{7^3}$.  Computing the coordinates of $m_\omega(1)$, $m_\omega(\omega)$ and $m_\omega(\omega^2)$ with respect to this same basis gives a matrix $S\in\GL_3(7)$ which is a Singer cycle.
Equivalently, one may construct $S$ as the companion matrix of a primitive polynomial of degree $3$ over $\F_7$, or verify directly that the resulting matrix has order $7^3-1$.
For the example used in the script, one obtains
\[
  S = \begin{pmatrix}
    0 & 0 & 3 \\
    1 & 0 & 0 \\
    0 & 1 & 1
  \end{pmatrix},
\]
and the computation confirms that this matrix has order $342 = 7^3-1$.
Hence $S$ is indeed a Singer cycle in $\GL_3(7)$; this matrix also appears explicitly in the console output.

We then construct the induced matrices of $S$ on $\Sym^k(V)$, where $V \cong \F_7^3$ is the natural module and $k = 2,3$.
This is done concretely via the tensor power $V^{\tensor k}$ and the symmetrisation operator: we work with the basis of $V^{\tensor k}$ indexed by $k$--tuples of $\{0,1,2\}$, apply the $k$--fold tensor product $S^{\tensor k}$, and restrict to the symmetric subspace spanned by symmetrised basis vectors.
The resulting matrices $S^{(k)} \in \GL(\Sym^k(V))$ have dimensions
\[
  \dim \Sym^2(V) = \binom{3+2-1}{2} = 6, \qquad
  \dim \Sym^3(V) = \binom{3+3-1}{3} = 10,
\]
as expected, and the script prints these dimensions together with an explicit list of the multi-indices labelling the basis of $\Sym^k(V)$.

Over $\F_{7^3}$, we compute the eigenvalues of $S^{(k)}$ and compare them with the theoretical eigenvalues
\[
  \lambda_{\mathbf{c}} = \omega^{\sum_i c_i 7^{i-1}}, \qquad \mathbf{c} = (c_1,c_2,c_3) \in \mathbb{Z}_{\ge 0}^3,\ \sum_i c_i = k.
\]
The experiments yield the following:

\begin{itemize}
  \item For $k = 2$: there are exactly $6$ distinct eigenvalues of $S^{(2)}$ on $\Sym^2(V)$, and they coincide with the $6$ values $\lambda_{\mathbf{c}}$ arising from the $6$ patterns $\mathbf{c}$ with $\sum c_i = 2$.
  \item For $k = 3$: there are exactly $10$ distinct eigenvalues of $S^{(3)}$ on $\Sym^3(V)$, and they coincide with the $10$ values $\lambda_{\mathbf{c}}$ arising from the $10$ patterns $\mathbf{c}$ with $\sum c_i = 3$.
\end{itemize}

In both cases the script also checks the multiplicities of eigenvalues.
For $k = 2$ it reports
\[
  \texttt{Number of eigenvalues returned (with multiplicity): 6}
\]
followed by
\[
  \texttt{All eigenvalues have multiplicity 1 (simple spectrum).}
\]
Similarly, for $k=3$ it reports
\[
  \texttt{Number of eigenvalues returned (with multiplicity): 10}
\]
and again
\[
  \texttt{All eigenvalues have multiplicity 1 (simple spectrum).}
\]
Thus, in these examples $S^{(2)}$ and $S^{(3)}$ both have simple spectrum on $\Sym^2(V)$ and $\Sym^3(V)$, respectively.
This is consistent with the fact that these modules are multiplicity-free for the diagonal torus, and provides concrete examples illustrating Theorem~\ref{thm:simple-spectrum} in the simplest non-trivial cases.

Finally, the script compares the sets of eigenvalues obtained from the actual matrices $S^{(k)}$ with the model set $\{\lambda_{\mathbf{c}} : \mathbf{c} \in \mathcal{B}_k',\, \sum_i c_i = k\}$.
In both cases it prints
\[
  \texttt{Size of set of eigenvalues (real)   :} \quad \dim \Sym^k(V),
\]
\[
  \texttt{Size of set of eigenvalues (model)  :} \quad |\mathcal{B}_k'|,
\]
followed by
\[
  \texttt{SUCCESS: Eigenvalues on Sym\^k(V) match the digit-vector model.}
\]
Since $|\mathcal{B}_k'| = \dim \Sym^k(V)$, this confirms that the distinct eigenvalues are in bijection with the weight patterns $\mathbf{c}$, exactly as predicted by Theorem~\ref{thm:distinct-eigenvalues-weak}.

As a typical example in the case $k=3$, the script reports an eigenvalue $\lambda = \omega^{147}$ for $S^{(3)}$.
Taking discrete logarithms and expanding $147$ in base~$7$ yields digits $(0,0,3)$.
This demonstrates that the eigenvalue indeed corresponds to the weight pattern $\mathbf{c} = (0,0,3)$, in perfect agreement with the formula in Section~\ref{sec:distinct-eigenvalues}.

\subsection{A toy reconstruction experiment for the rewriting step}

The conditional framework of Section~\ref{sec:algorithm} relies, in Step~4, on recovering the underlying projective matrix $A(g)$ of an element $g$ on the natural module $V$ from its action on a polynomial module such as a symmetric or exterior power.
For a specific functor, this can often be done by identifying suitable matrix entries of $M_W(g)$ which are given by low-degree monomials in the entries of $A(g)$ and solving the resulting equations.
To illustrate this mechanism in a particularly simple setting, we implemented a toy model for the symmetric square in dimension $d=2$.

Let $V = \F_q^2$ with $q$ odd, and let $e_0, e_1$ be the standard basis of $V$. We identify $\Sym^2(V)$ with the $3$--dimensional space spanned by
\[
  v_1 = e_0 \tensor e_0, \quad
  v_2 = e_0 \tensor e_1 + e_1 \tensor e_0, \quad
  v_3 = e_1 \tensor e_1.
\]
For a matrix
\[
  A = \begin{pmatrix} a & b \\ c & d \end{pmatrix} \in \GL_2(q),
\]
we let $A$ act on $V^{\tensor 2}$ diagonally via
\[
  A \cdot (u \tensor w) = (Au) \tensor (Aw).
\]
In particular,
\[
  A e_0 = a e_0 + c e_1, \qquad A e_1 = b e_0 + d e_1.
\]
We briefly spell out the computation for $v_1$; the other cases are analogous. Using linearity of the tensor product,
\[
  A \cdot v_1
  \;=\; A \cdot (e_0 \tensor e_0)
  \;=\; (A e_0) \tensor (A e_0)
  \;=\; (a e_0 + c e_1) \tensor (a e_0 + c e_1),
\]
so
\[
  (a e_0 + c e_1) \tensor (a e_0 + c e_1)
  = a^2 (e_0 \tensor e_0)
    + a c (e_0 \tensor e_1 + e_1 \tensor e_0)
    + c^2 (e_1 \tensor e_1)
  = a^2 v_1 + a c\, v_2 + c^2 v_3.
\]
A similar expansion for $v_2$ and $v_3$ yields
\[
  A \cdot v_3 = b^2 v_1 + b d\, v_2 + d^2 v_3,
\]
\[
  A \cdot v_2 = 2 a b\, v_1 + (a d + b c)\, v_2 + 2 c d\, v_3.
\]
Thus, with respect to the basis $\{v_1,v_2,v_3\}$, the matrix $\Sym^2(A)$ has columns
\[
  \Sym^2(A) =
  \begin{pmatrix}
    a^2 & 2ab & b^2 \\
    ac  & ad+bc & bd \\
    c^2 & 2cd & d^2
  \end{pmatrix}.
\]

The script contains a function that implements this formula and, given a matrix $M \in \GL_3(q)$ of this form, performs a brute-force search over all $A \in \GL_2(q)$ to find those satisfying $\Sym^2(A) = M$.
Since the field is finite and we restrict ourselves to $q=7$, this is perfectly feasible for experimental purposes.
The aim is to verify that, generically, one can recover $A$ from $\Sym^2(A)$ up to the expected projective ambiguity.
For exact equality of matrices $\Sym^2(A')=\Sym^2(A)$, the scalar ambiguity is restricted by the kernel of the symmetric-square functor on scalars; in the projective setting this is precisely the natural ambiguity relevant for rewriting.

Concretely, for $q = 7$ the script runs a small number of random trials.
For each trial it chooses a random $A \in \GL_2(7)$, computes $M = \Sym^2(A)$, and then searches for all $A' \in \GL_2(7)$ with $\Sym^2(A') = M$.
Among the candidates it checks whether there is an $A'$ which is a scalar multiple of $A$.
In all trials this is the case.
For example, in one run the script reports
\[
  A =
  \begin{pmatrix}
    6 & 2 \\
    2 & 4
  \end{pmatrix}, \qquad
  A' =
  \begin{pmatrix}
    1 & 5 \\
    5 & 3
  \end{pmatrix},
\]
with the property that $\Sym^2(A') = \Sym^2(A)$.
A quick check in $\F_7$ shows that $A' = \lambda A$ with $\lambda = 6$, since
\[
  6 \cdot A' =
  6 \begin{pmatrix} 1 & 5 \\[1pt] 5 & 3 \end{pmatrix}
  = \begin{pmatrix} 6 & 30 \\ 30 & 18 \end{pmatrix}
  \equiv \begin{pmatrix} 6 & 2 \\ 2 & 4 \end{pmatrix} \pmod 7.
\]
In other trials the script similarly finds that $A$ is determined up to the expected projective ambiguity by the matrix $\Sym^2(A)$.

This toy example thus provides an explicit, concrete illustration of the functor-specific inversion step assumed in Theorem~\ref{thm:rewriting}.  In that theorem, $W$ may be a more complicated tensor product of polynomial modules, and the extraction of $A(g)$ from $M_W(g)$ uses the known Schur functor structure of each factor $L(\lambda^{(t)})$, rather than brute force.  The small-scale experiment for $\Sym^2(V)$ in dimension $2$ confirms the underlying principle that the entries of $M_W(g)$ are polynomial expressions in the entries of $A(g)$ and can, in favourable functor-specific situations, determine the natural projective action.

\subsection{A genuine tensor-product rewriting example over \texorpdfstring{$\F_9$}{F9}}
\label{sec:genuine-tensor-example}

We now give a hand-computed example in which the representation is a genuine tensor product and the full tensor product is multiplicity-free for the diagonal torus.  This example is included to illustrate explicitly the functor-specific inversion step in a case which is not a symmetric square.

Let
\[
  \F_9=\F_3(\alpha),\qquad \alpha^2=2.
\]
Let \(V=\F_9^3\) with basis \(e_1,e_2,e_3\), and consider the polynomial \(\GL_3\)-module
\[
  W=L(1,0,0)\tensor L(3,0,0).
\]
Since the characteristic is \(3\), the second factor \(L(3,0,0)\) is the first Frobenius twist of the natural module.  Thus, for \(A\in \GL_3(9)\), the induced action on \(W\) is
\[
  \rho(A)=A\tensor A^{(3)},
\]
where \(A^{(3)}\) denotes the matrix obtained by cubing each entry of \(A\).  The polynomial degrees of the two factors are \(1\) and \(3\), so
\[
  K=1+3=4<8=q-1.
\]

The weights of \(W\) are
\[
  \varepsilon_i+3\varepsilon_j,\qquad 1\le i,j\le 3.
\]
These weights are pairwise distinct.  Indeed, if
\[
  \varepsilon_i+3\varepsilon_j=\varepsilon_{i'}+3\varepsilon_{j'},
\]
then equality of the coefficients of the basis characters \(\varepsilon_1,\varepsilon_2,\varepsilon_3\) forces \(i=i'\) and \(j=j'\).  Hence \(W\) is multiplicity-free for the diagonal torus.  With respect to the tensor basis
\[
  w_{ij}=e_i\tensor e_j^{(1)},\qquad 1\le i,j\le 3,
\]
where \(e_j^{(1)}\) denotes the corresponding basis vector in the Frobenius-twisted factor, each weight space is the one-dimensional space \(\F_9 w_{ij}\).

We next choose an explicit Singer cycle.  Let
\[
  m(X)=X^3+(1+\alpha)X^2+(1+\alpha)\in \F_9[X].
\]
This polynomial is primitive of degree \(3\) over \(\F_9\).  If \(\omega\) is the image of \(X\) in \(\F_9[X]/(m(X))\), then \(\omega\) generates \(\F_{9^3}^{\times}\), whose order is
\[
  9^3-1=728.
\]
With respect to the basis \(1,\omega,\omega^2\) of \(\F_{9^3}\) over \(\F_9\), multiplication by \(\omega\) is represented by
\[
  S=
  \begin{pmatrix}
    0&0&2+2\alpha\\
    1&0&0\\
    0&1&2+2\alpha
  \end{pmatrix}.
\]
Thus \(S\) is a genuine Singer cycle in \(\GL_3(9)\).

Over \(\F_{9^3}\), choose a Singer eigenbasis \(u_1,u_2,u_3\) for the natural module, with
\[
  S u_i=\ell_i u_i,\qquad
  \ell_i=\omega^{9^{i-1}},\qquad i=1,2,3.
\]
On the Frobenius-twisted factor, the eigenvalue on \(u_j^{(1)}\) is
\[
  \ell_j^3=\omega^{3\cdot 9^{j-1}}.
\]
Therefore the eigenvalue of \(\rho(S)=S\tensor S^{(3)}\) on
\[
  u_i\tensor u_j^{(1)}
\]
is
\[
  \ell_i\ell_j^3
  =
  \omega^{E_{ij}},
  \qquad
  E_{ij}=9^{i-1}+3\cdot 9^{j-1}.
\]
Equivalently, \(E_{ij}\) has base-\(9\) digit vector \(\mathbf c_{ij}\) given by
\[
  \mathbf c_{ij}=
  \begin{cases}
    4\mathbf e_i, & i=j,\\
    \mathbf e_i+3\mathbf e_j, & i\ne j,
  \end{cases}
\]
where \(\mathbf e_1,\mathbf e_2,\mathbf e_3\) are the standard coordinate vectors.  Explicitly,
\[
\begin{array}{c|c|c}
\text{eigenline} & \mathbf c_{ij} & E_{ij}\\ \hline
u_1\tensor u_1^{(1)} & (4,0,0) & 4\\
u_1\tensor u_2^{(1)} & (1,3,0) & 28\\
u_1\tensor u_3^{(1)} & (1,0,3) & 244\\
u_2\tensor u_1^{(1)} & (3,1,0) & 12\\
u_2\tensor u_2^{(1)} & (0,4,0) & 36\\
u_2\tensor u_3^{(1)} & (0,1,3) & 252\\
u_3\tensor u_1^{(1)} & (3,0,1) & 84\\
u_3\tensor u_2^{(1)} & (0,3,1) & 108\\
u_3\tensor u_3^{(1)} & (0,0,4) & 324
\end{array}
\]
All these exponents lie in the interval \([0,727]\) and are distinct.  This is exactly the mechanism of Lemma~\ref{lem:injectivity}: the digit coordinates are bounded by \(K=4<8=q-1\), so distinct digit vectors determine distinct residues modulo \(9^3-1\).  Hence \(\rho(S)\) has simple spectrum on
\[
  W\tensor_{\F_9}\F_{9^3},
\]
and the Singer eigenvalues isolate the nine eigenlines \( \F_{9^3}(u_i\tensor u_j^{(1)})\).

We now demonstrate the reconstruction step for a concrete group element.  Let
\[
  A=
  \begin{pmatrix}
    1&\alpha&0\\
    0&1&1\\
    1&0&\alpha
  \end{pmatrix}
  \in \GL_3(9).
\]
Its determinant is
\[
  \det(A)=2\alpha\ne 0.
\]
Since the Frobenius automorphism of \(\F_9/\F_3\) has order \(2\), we have
\[
  \alpha^3=2\alpha
\]
and hence
\[
  A^{(3)}=
  \begin{pmatrix}
    1&2\alpha&0\\
    0&1&1\\
    1&0&2\alpha
  \end{pmatrix}.
\]
With respect to the ordered tensor basis
\[
  w_{11},w_{12},w_{13},w_{21},w_{22},w_{23},w_{31},w_{32},w_{33},
\]
the matrix of \(\rho(A)=A\tensor A^{(3)}\) is the following \(9\times 9\) block matrix:
\[
\rho(A)=
\begin{pmatrix}
1&2\alpha&0&\alpha&1&0&0&0&0\\
0&1&1&0&\alpha&\alpha&0&0&0\\
1&0&2\alpha&\alpha&0&1&0&0&0\\
0&0&0&1&2\alpha&0&1&2\alpha&0\\
0&0&0&0&1&1&0&1&1\\
0&0&0&1&0&2\alpha&1&0&2\alpha\\
1&2\alpha&0&0&0&0&\alpha&1&0\\
0&1&1&0&0&0&0&\alpha&\alpha\\
1&0&2\alpha&0&0&0&\alpha&0&1
\end{pmatrix}.
\]
The block structure is visible: the \((r,s)\)-block is \(a_{rs}A^{(3)}\).  In this example \(a_{11}=1\), so the upper-left \(3\times 3\) block is exactly \(A^{(3)}\).  Applying the Frobenius automorphism once more to this block recovers
\[
  (A^{(3)})^{(3)}=A.
\]
Thus the natural matrix \(A\) is recovered from the tensor-product matrix \(\rho(A)\).

More generally, if \(M=\rho(A)=A\tensor A^{(3)}\) is given in the tensor basis and the \((r,s)\)-block is nonzero, then that block has the form
\[
  M_{rs}=a_{rs}A^{(3)}.
\]
Cubing entries gives
\[
  M_{rs}^{(3)}=a_{rs}^{\,3}A,
\]
which is a scalar multiple of \(A\).  Hence the projective class \([A]\in \PGL_3(9)\) is recovered from any nonzero block of \(M\).  This gives an explicit functor-specific inversion for the tensor product
\[
  L(1,0,0)\tensor L(3,0,0).
\]

This example illustrates the complete mechanism used in Theorem~\ref{thm:rewriting}.  The Singer cycle gives one-dimensional labelled eigenlines because the tensor product is multiplicity-free and \(K<q-1\).  Once the labelled eigenbasis has been identified with the tensor basis, the induced matrix \(A\tensor A^{(3)}\) determines the natural projective matrix \(A\) by the explicit block inversion above.

\subsection{Discussion}

Although Theorems~\ref{thm:distinct-eigenvalues-weak} and~\ref{thm:simple-spectrum} are proved theoretically, the experiments above provide a useful independent check of the eigenvalue-labelling mechanism and of the reconstruction philosophy in a few small but representative cases.
They show that:

\begin{itemize}
  \item the base-$q$ injectivity lemma (Lemma~\ref{lem:injectivity}) behaves as predicted for various small triples $(q,d,C)$ with $C<q-1$;
  \item the exponent formula $E(\nu) = \sum_i c_i(\nu) q^{i-1}$ correctly encodes the eigenvalues of a Singer cycle on symmetric powers of the natural module, and the digits recovered from discrete logarithms match the expected weight patterns;
\item the exponent-model labelling strategy scales efficiently for the parameters $(q,d)=(2^{16},10)$, identifying digit vectors without constructing matrices over $\mathbb{F}_{2^{160}}$, as demonstrated in Section~\ref{sec:experiments-model};
  \item in multiplicity-free modules such as $\Sym^k(V)$ for $k \le 3$ and $(q,d) = (7,3)$, the spectrum of a Singer cycle is indeed simple, and the map from weight patterns to eigenvalues is a bijection;
  \item in the toy example $d=2$ with $W=\Sym^2(V)$ over $\F_7$, the matrix $\Sym^2(A)$ determines $A$ up to the expected projective ambiguity for random choices of $A \in \GL_2(7)$, in line with the reconstruction philosophy discussed in Section~\ref{sec:algorithm}.
\item in the genuine tensor-product example $W = L(1,0,0) \tensor L(3,0,0)$ over $\F_9$, the base-$9$ expansion explicitly separates the weight spaces, and the Kronecker structure allows immediate recovery of the projective natural matrix, proving that the conditional framework successfully executes rewriting on a strict multiplicity-free tensor product.
\end{itemize}

From the algorithmic point of view, these examples illustrate on a small scale how the spectral labelling produced by a Singer cycle can be used to organise basis vectors by weight patterns and thereby reduce rewriting to a reconstruction problem on the natural module.
They do not constitute a proof of a fully uniform reconstruction theorem for arbitrary polynomial highest weights.
Rather, they provide evidence that the framework developed in Section~\ref{sec:algorithm} is effective in low-degree cases and in concrete families such as symmetric powers.

\vspace*{1cm}

\subsection*{Funding}
This research received no external funding.

\subsection*{Conflict of interest}
The author declares that there is no conflict of interest.

\subsection*{Data Availability} 
Data sharing is not applicable to this article as no datasets were generated or analyzed during the current study. The illustrative \textsf{SageMath} subroutines and all explicit mathematical computations supporting the findings are provided directly within the manuscript.


\begin{thebibliography}{99}

\bibitem{BealsLGNPS2003}
R.~Beals, C.~R.~Leedham-Green, A.~C.~Niemeyer, C.~E.~Praeger and A.~Seress,
\newblock \textit{A black-box group algorithm for recognizing finite symmetric and alternating groups, I},
\newblock Trans.\ Amer.\ Math.\ Soc. \textbf{355} (2003), 2097--2113.

\bibitem{Brooksbank2003}
P.~A.~Brooksbank,
\newblock \textit{Constructive recognition of classical groups in their natural representation},
\newblock J.\ Symbolic Comput. \textbf{35} (2003), 195--239.

\bibitem{Corr2015}
B.~P.~Corr,
\newblock \textit{A Las Vegas rewriting algorithm for the symmetric square representation of classical groups},
\newblock Preprint, 2015, Arxiv:1507.05671.

\bibitem{DietrichLGO2015}
H.~Dietrich, C.~R.~Leedham-Green and E.~A.~O'Brien,
\newblock \textit{Effective black-box constructive recognition of classical groups},
\newblock J.\ Algebra \textbf{421} (2015), 460--492.

\bibitem{GulAnk2016}
K.~G\"ul and N.~Ankaral{\i}o\u{g}lu,
\newblock \textit{On the twisted modules for finite matrix groups},
\newblock Turkish J.\ Math. \textbf{40} (2016), 191--200.

\bibitem{Jantzen2003}
J.~C.~Jantzen,
\newblock {\em Representations of Algebraic Groups}, 2nd ed.,
\newblock Mathematical Surveys and Monographs, Vol.~107, American Mathematical Society, Providence, RI, 2003.

\bibitem{Luebeck2001}
F.~L\"ubeck,
\newblock \textit{Small degree representations of finite Chevalley groups in defining characteristic},
\newblock LMS J.\ Comput.\ Math. \textbf{4} (2001), 135--169.

\bibitem{MagaardOBrienSeress2008}
K.~Magaard, E.~A.~O'Brien and \'{A}.~Seress,
\newblock \textit{Recognition of small dimensional representations of general linear groups},
\newblock J.\ Aust.\ Math.\ Soc. \textbf{85} (2008), 229--250.

\end{thebibliography}
\end{document}